\documentclass[12pt, a4paper]{amsart}

\usepackage{a4}
\usepackage{amssymb}
\usepackage{amsmath}
\usepackage{amsthm}
\usepackage{amstext}
\usepackage{amscd}
\usepackage{latexsym}
\usepackage{graphics}
\usepackage{color}
\usepackage[all]{xy}

\input{xy}
\xyoption{all}

\usepackage[colorlinks, pagebackref]{hyperref}
\hypersetup{
  colorlinks=true,
  citecolor=blue,
  linkcolor=blue,
  urlcolor=blue}

\makeatletter
\def\@tocline#1#2#3#4#5#6#7{\relax
  \ifnum #1>\c@tocdepth 
  \else
    \par \addpenalty\@secpenalty\addvspace{#2}%
    \begingroup \hyphenpenalty\@M
    \@ifempty{#4}{%
      \@tempdima\csname r@tocindent\number#1\endcsname\relax
    }{%
      \@tempdima#4\relax
    }%
    \parindent\z@ \leftskip#3\relax \advance\leftskip\@tempdima\relax
    \rightskip\@pnumwidth plus4em \parfillskip-\@pnumwidth
    #5\leavevmode\hskip-\@tempdima
      \ifcase #1
       \or\or \hskip 2em \or \hskip 2em \else \hskip 3em \fi%
      #6\nobreak\relax
    \dotfill\hbox to\@pnumwidth{\@tocpagenum{#7}}\par
    \nobreak
    \endgroup
  \fi}
\makeatother

\theoremstyle{plain}

\newtheorem{introthm}{Theorem}[]
\newtheorem{theorem}{Theorem}[section]

\newtheorem{lemma}[theorem]{Lemma}
\newtheorem{corollary}[theorem]{Corollary}
\newtheorem{proposition}[theorem]{Proposition}

\theoremstyle{definition}
\newtheorem{remark}[theorem]{Remark}

\newtheorem{notation}[theorem]{Notation}

\newtheorem{definition}[theorem]{Definition}

\newtheorem{examples}[theorem]{Examples}
\newtheorem{hypothesis}[theorem]{Hypothesis}

\newcommand{\tensor}{\otimes}

\newcommand{\Hom}{{\rm Hom}}

\newcommand{\Spec}{{\rm Spec \,}}

\newcommand{\sF}{{\mathcal F}}

\newcommand{\sH}{{\mathcal H}}

\newcommand{\sS}{{\mathcal S}}

\newcommand{\sU}{{\mathcal U}}

\newcommand{\sX}{{\mathcal X}}
\newcommand{\sY}{{\mathcal Y}}

\newcommand{\A}{{\mathbb A}}

\newcommand{\G}{{\mathbb G}}

\newcommand{\N}{{\mathbb N}}

\renewcommand{\P}{{\mathbb P}}

\newcommand{\Sing}{{\rm Sing_*^{\A^1}}}

\def\<{\langle}
\def\>{\rangle} 
\def\-{\overline} 
\def\~{\widetilde}
\def\^{\widehat}

\begin{document}

\title{$\A^1$-connectedness in reductive algebraic groups}

\author{Chetan Balwe}
\address{Department of Mathematics, Indian Institute of Science Education and Research (IISER), Knowledge City, Sector-81, Mohali 140306, India.}

\email{cbalwe@iisermohali.ac.in}

\author{Anand Sawant}
\address{Mathematisches Institut, Ludwig-Maximilians Universit\"at, Theresienstr. 39, 
D-80333 M\"unchen, Germany.}
\email{sawant@math.lmu.de}
\date{}

\begin{abstract}
Using sheaves of $\A^1$-connected components, we prove that the Morel-Voevodsky singular construction on a reductive algebraic group fails to be $\A^1$-local if the group does not satisfy suitable isotropy hypotheses.  As a consequence, we show the failure of $\A^1$-invariance of torsors for such groups on smooth affine schemes over infinite perfect fields.  We also characterize $\A^1$-connected reductive algebraic groups over a field of characteristic $0$.
\end{abstract}

\maketitle

\section{Introduction}
\label{section Introduction}

Let us fix a base field $k$ and let $\sH(k)$ denote the $\A^1$-homotopy category of schemes over $k$ developed by Morel and Voevodsky \cite{Morel-Voevodsky}.  This category is constructed by first enlarging the category of smooth (finite type, separated) schemes over $k$ to the category of simplicial sheaves on the big Nisnevich site of smooth schemes over $k$ and then taking a suitable localization.  Given a smooth scheme $X$ over $k$, one may ask if the set of morphisms $\Hom_{\sH(k)}(U, X)$ has a geometric description, at least when $U$ is a smooth henselian local scheme.  In particular, one may ask if the set $\Hom_{\sH(k)}(U, X)$ is in bijection with the equivalence classes of morphisms of schemes $U \to X$ by \emph{naive} $\A^1$-homotopies.  This question is closely related to the behaviour of the Morel-Voevodsky singular construction $\Sing X$ (for precise definitions, see Section \ref{section preliminaries A1-connectedness}).  More precisely, the above question has an affirmative answer if $\Sing X$ is \emph{$\A^1$-local}, that is, discrete as an object of $\sH(k)$ (see Definition \ref{definition A1-local}).  However, there exist (even smooth, projective) varieties $X$ for which $\Sing X$ is not $\A^1$-local \cite[\textsection 4.1]{Balwe-Hogadi-Sawant}.

In this paper, we will study this question for a reductive algebraic group $G$ over an infinite perfect field $k$.  Under a suitable isotropy hypothesis on $G$, it was shown by Asok, Hoyois and Wendt \cite{Asok-Hoyois-Wendt-2} that $\Sing G$ is $\A^1$-local.  This isotropy hypothesis on reductive algebraic $k$-groups was introduced in \cite{Raghunathan-1989}:
\[
\begin{split}
(\ast)~ & \text{Every almost $k$-simple component of the derived group $G_{\rm der}$ of $G$ contains}\\ 
& \text{a $k$-subgroup scheme isomorphic to $\G_m$.}
\end{split}
\]
Asok, Hoyois and Wendt obtain the $\A^1$-locality of $\Sing G$ for a reductive algebraic group satisfying the hypothesis $(\ast)$ by showing affine homotopy invariance of Nisnevich locally trivial torsors under such groups.  More precisely, they show that for every smooth affine scheme $U$ over $k$ and a group $G$ satisfying $(\ast)$, the natural map 
\[
H^1_{\rm Nis}(U, G) \to H^1_{\rm Nis}(U \times \A^n, G)
\]
is a bijection, for every $n \geq 0$.  Special cases of this result (such as the case $G = GL_n$ \cite{Lindel} and the case where $G$ satisfies $(\ast)$ and $U = \Spec k$ \cite{Raghunathan-1989}) were known much earlier.  

Examples of anisotropic groups $G$ for which $\Sing G$ fails to be $\A^1$-local were obtained in \cite{Balwe-Sawant}.  For these examples, the failure of affine homotopy invariance of $G$-torsors was noted in \cite{Asok-Hoyois-Wendt-2}.  Explicit examples of failure of affine homotopy invariance of the presheaf $H^1_{\rm Nis}(-, G)$ were already known in many cases, see \cite{Ojanguren-Sridharan}, \cite{Parimala} for the first examples; and \cite[Theorem B]{Raghunathan-1989} for a general statement excluding certain groups that are not of classical type.  See Section \ref{subsection failure A1-invariance} for more details.  In this paper, we generalize these results as follows:

\begin{introthm}[Theorem \ref{theorem equivalence}]
\label{intro theorem 1}
Let $G$ be a reductive algebraic group over an infinite perfect field $k$.  Then the following conditions are equivalent: 
\begin{itemize}
\item[(1)] $\Sing G$ is $\A^1$-local;
\item[(2)] $G$ satisfies the isotropy hypothesis $(\ast)$ (Hypothesis \ref{hypothesis isotropy});
\item[(3)] The presheaf $H^1_{\rm Nis}(-, G)$ is $\A^1$-invariant on smooth affine schemes over $k$.
\end{itemize}
\end{introthm}

Thus, Theorem \ref{intro theorem 1} compares a motivic statement, a group-theoretic statement and a cohomological statement.  In view of \cite{Asok-Hoyois-Wendt-2}, the main results of this paper are Theorem \ref{theorem failure of A1-locality} and Proposition \ref{proposition failure of A1-invariance}.  Although Theorem \ref{intro theorem 1} addresses all reductive groups that fail to satisfy the isotropy hypothesis $(\ast)$ uniformly, our proof of failure of affine homotopy invariance for groups that do not satisfy hypothesis $(\ast)$ is existential.  The main ingredient in the proof of failure of $\A^1$-locality of the singular construction for such groups is the behaviour of $\A^1$-locality of the singular construction under central isogenies, which is described in Section \ref{subsection failure A1-locality}.

In \cite{Balwe-Sawant}, it was shown that sections over fields of the sheaf of $\A^1$-connected components of a semisimple simply connected group agree with the group of its $R$-equivalence classes.  However, this is not the case if we drop the hypothesis of simple connectedness.  This can be seen via our following characterization of $\A^1$-connected reductive algebraic groups.

\begin{introthm}[Theorem \ref{theorem reductive connectedness}]
\label{intro theorem 2}
Let $G$ be a reductive algebraic group over a field $k$ of characteristic $0$.  Then $G$ is $\A^1$-connected if and only if $G$ is semisimple, simply connected and every almost $k$-simple factor of $G$ is $R$-trivial. 
\end{introthm}

In particular, this shows that one cannot have $\A^1$-connected components of a non-simply connected group agree with its $R$-equivalence classes in general.  Indeed, any split semisimple group is a rational variety \cite[V.15.8]{Borel} and hence $R$-trivial; however, its $\A^1$-connected components cannot be trivial unless the group is simply connected.  As a consequence of Theorem \ref{intro theorem 2}, we observe that $\A^1$-connected components of a semisimple group with an $\A^1$-connected, simply connected central cover form a sheaf of abelian groups (see Proposition \ref{proposition pi0 abelian}).

We now outline the contents of this paper.  In Section \ref{section preliminaries A1-connectedness}, we recollect preliminaries about $\A^1$-connectedness.  In Section \ref{section preliminaries algebraic groups}, we recall basic notions about reductive algebraic groups and describe $\A^1$-connected components of semisimple, simply connected algebraic groups as an immediate consequence of results of \cite{Balwe-Sawant}.  Section \ref{section failure} is devoted to the proof of Theorem \ref{intro theorem 1}.  In Section \ref{section characterization of A1-connectedness}, we characterize $\A^1$-connected reductive algebraic groups over a field of characteristic $0$.  We then obtain the abelian-ness of the sheaf of $\A^1$-connected components of certain algebraic groups as an application.

\section{Preliminaries on \texorpdfstring{$\A^1$}{A1}-connected components of schemes}
\label{section preliminaries A1-connectedness}
In this section, we briefly recall some definitions from the $\A^1$-homotopy theory and some basic properties, particularly regarding $\A^1$-connectedness. We will begin by briefly reviewing the construction of the \emph{$\A^1$-homotopy category} from \cite{Morel-Voevodsky}.  

Let $k$ be a field.  Let $Sm/k$ denote the big Nisnevich site of smooth, separated, finite-type schemes over $k$. We begin with the category of simplicial sheaves over $Sm/k$.
A morphism $\sX \to \sY$ of simplicial sheaves is a \emph{local weak equivalence} if it induces a weak equivalence of stalks $\sX_x \to \sY_x$ at every point $x$ of the site. 
The \emph{local injective model structure} on this category is the one in which the morphism of simplicial sheaves is a cofibration (resp. a weak equivalence) if and only if it is a monomorphism (resp. a local weak equivalence). The corresponding homotopy category is called the \emph{simplicial homotopy category} and is denoted by $\sH_s(k)$. The left Bousfield localization of the local injective model structure with respect to the collection of all projection morphisms $\sX \times \A^1 \to \sX$, as $\sX$ runs over all simplicial sheaves, is called the \emph{$\A^1$-model structure}. The corresponding homotopy category is called the \emph{$\A^1$-homotopy category} and is denoted by $\sH(k)$.  

\begin{definition}
\label{definition A1-local}
A simplicial sheaf $\sX$ on $Sm/k$ is said to be \emph{$\A^1$-local} if for any $U \in Sm/k$, the projection map $U \times \A^1 \to U$ induces a bijection 
\[
\Hom_{\sH_s(k)}(U, \sX) \to \Hom_{\sH_s(k)}(U \times \A^1, \sX).
\] 
Following standard conventions, an $\A^1$-local scheme will be called \emph{$\A^1$-rigid}.  A scheme $X \in Sm/k$ is $\A^1$-rigid if for every $U \in Sm/k$, any morphism $h: U \times \A^1 \to X$ factors through the projection map $U \times \A^1 \to U$.
\end{definition}

\begin{examples}
Curves of genus $\geq 1$, abelian varieties and algebraic tori are some examples of $\A^1$-rigid schemes.
\end{examples}

We now recall the singular construction $\Sing$ in $\A^1$-homotopy theory defined by Morel-Voevodsky (see \cite[p.87]{Morel-Voevodsky}).  For a simplicial sheaf $\sX$ on $Sm/k$, define $\Sing \sX$ to be the simplicial sheaf given by
\[
(\Sing \sX)_n = \underline{\Hom}(\Delta_n,\sX_n), 
\]
\noindent where $\Delta_{\bullet}$ denotes the simplicial sheaf  
\[
\Delta_n  = \Spec\left(\frac{k[x_0,...,x_n]}{(\sum_ix_i=1)}\right)
\]
\noindent with natural face and degeneracy maps analogous to the ones on topological simplices. The functor $\Sing$ commutes with limits; in particular, with products. Also, there exists a natural transformation $Id \to \Sing$ such that for any simplicial sheaf $\sX$, the morphism $\sX \to \Sing(\sX)$ is an $\A^1$-weak equivalence.

There exists an \emph{$\A^1$-localization} endofunctor (\cite[\textsection 2, Theorem 1.66 and p.107]{Morel-Voevodsky}) on the simplicial homotopy category $\sH_s(k)$, denoted by $L_{\A^1}$, such that for every simplicial sheaf $\sX$, the simplicial sheaf $L_{\A^1}(\sX)$ is $\A^1$-local.  
  In \cite[\textsection 2, Theorem 1.66 and p. 107]{Morel-Voevodsky}, an explicit description of $L_{\A^1}$ is given as follows:
\[
L_{\A^1} = Ex \circ (Ex \circ \Sing)^{\N} \circ Ex,
\]
\noindent where $Ex$ denotes a simplicial fibrant replacement functor on $\sH_s(k)$. There exists a natural transformation $Id \to L_{\A^1}$ which factors through the natural transformation $Id \to \Sing$ mentioned above. For any object $\sX$, the morphism $\sX \to L_{\A^1}(\sX)$ is an $\A^1$-weak equivalence. 

\begin{notation}
Given a simplicial sheaf of sets $\sX$ on $Sm/k$, we will denote by $\pi_0(\sX)$ the presheaf on $Sm/k$ that associates with $U \in Sm/k$ the coequalizer of the diagram $\sX_1(U) \rightrightarrows \sX_0(U)$, where the maps are the face maps coming from the simplicial data of $\sX$.  We will denote by $\pi_0^s(\sX)$ the Nisnevich sheafification of the presheaf $\pi_0(\sX)$.

Now, let $n \geq 1$ be an integer and let $(\sX, x)$ be a pointed simplicial sheaf of sets on $Sm/k$.  For $U \in Sm/k$, we will denote by $U_+$ the scheme $U \coprod \Spec k$, pointed at the added basepoint $\Spec k$.  We will denote by $\pi_n^s(\sX,x)$ the Nisnevich sheafification of the presheaf (of groups) on $Sm/k$ that associates with $U \in Sm/k$ the group $\Hom_{\sH_s(k)}(\Sigma_n^s U_{+}, (\sX,x))$ of simplicial homotopy classes of pointed maps from the simplicial $n$-fold suspension of the pointed scheme $U_{+}$ into $(\sX, x)$.

We caution the reader that this notation is not to be confused with the similar notation used for the sheaves of \emph{stable homotopy groups}.  Although this choice of notation is unfortunate, we use it here nevertheless in order to be consistent with the notation in \cite{Balwe-Hogadi-Sawant} and \cite{Balwe-Sawant}.
\end{notation}

\begin{definition}
\label{definition-S}
Let $\sX$ be a simplicial sheaf on $Sm/k$.  The sheaf of \emph{$\A^1$-chain connected components} of $\sX$ is defined by
\[
\sS(\sX) := \pi_0^s(\Sing \sX).
\]  
\end{definition}

\begin{remark}
Let $X$ be a scheme over $k$.  For any smooth scheme $U$ over $k$, we say that two morphisms $f,g: U \to X$ are \emph{$\A^1$-homotopic} if there exists a morphism $h: U \times \A^1 \to X$ such that $h|_{U \times \{0\}} = f$ and $h|_{U \times \{1\}} = g$.  We say that $f,g: U \to X$ are \emph{$\A^1$-chain homotopic} if there exists a finite sequence $f_0=f, \ldots, f_n=g$ such that $f_i$ is $\A^1$-homotopic to $f_{i+1}$, for all $i$.  Clearly, $\A^1$-chain homotopy is an equivalence relation. It is easy to see that $\sS(X)$ is the sheafification in the Nisnevich topology of the presheaf on $Sm/k$ that associates with every smooth scheme $U$ over $k$ the set of equivalence classes in $X(U)$ under the relation of $\A^1$-chain homotopy.  
\end{remark}

\begin{definition}
\label{definition pi0A1}
Let $\sX$ be a simplicial sheaf on $Sm/k$.  The sheaf of \emph{$\A^1$-connected components} of $\sX$ is defined by 
\[
\pi_0^{\A^1}(\sX) := \pi_0^s(L_{\A^1}(\sX)).
\]
For any smooth scheme $U$ over $k$, we will say that $f,g \in \sX(U)$ are \emph{$\A^1$-equivalent} if they map to the same element of $\pi_0^{\A^1}(\sX)(U)$.  We say that $\sX$ is \emph{$\A^1$-connected} if $\pi_0^{\A^1}(\sX) \simeq \ast$, the trivial point sheaf.
\end{definition}

There is a canonical epimorphism $\sS(\sX) \to \pi_0^{\A^1}(\sX)$ \cite[\textsection 2, Corollary 3.22, p. 94]{Morel-Voevodsky}. This epimorphism is an isomorphism if $\Sing \sX$ is $\A^1$-local.

\begin{definition}
\label{definition A1-homotopy sheaves}
Let $(\sX, x)$ be a pointed simplicial sheaf on $Sm/k$ (that is, $x$ is a morphism $\Spec k \to \sX$).  For every integer $n \geq 1$, the $n$th \emph{$\A^1$-homotopy sheaf} of $\sX$ with basepoint $x$ is defined by 
\[
\pi_n^{\A^1}(\sX, x) := \pi_n^s(L_{\A^1}\sX, x),
\]
where $L_{\A^1}(\sX)$ is pointed by $\Spec k \xrightarrow{x} \sX \to L_{\A^1}(\sX)$, which we continue to denote by $x$.  We will always suppress base-points for the sake of brevity, when the base-point is understood from notation.
\end{definition}

The main difficulty in the study of $\pi_0^{\A^1}$ of schemes is that the explicit description of the $\A^1$-localization functor is cumbersome to handle.  However, in the cases when $\pi_0^{\A^1}$ of a scheme is \emph{$\A^1$-invariant}, it can be studied with geometric methods using results of \cite{Balwe-Hogadi-Sawant}.

The notion of \emph{Weil restriction} of a simplicial sheaf will be very useful in what follows.  We briefly recall it here.  Let $F/k$ be a finite field extension and let $f: \Spec F \to \Spec k$ denote the morphism corresponding to the inclusion $k \hookrightarrow F$. The pushforward functor $f_*$ from the category of simplicial sheaves on $Sm/F$ into the category of simplicial sheaves on $Sm/k$ is defined by 
\[
f_*(\sX)(U) = \sX(U \times_{\Spec k} \Spec F).
\] 
If $Ex$ denotes a simplicial fibrant replacement functor on the category of simplicial sheaves over $Sm/k$, one can show that the functor $f_* \circ Ex$ preserves simplicial weak equivalences. Thus, it induces a functor $\mathbf{R}f_*: \sH_s(F) \to \sH_s(k)$, which is the \emph{right derived functor} of $f_*$.  We recall that the functor $\mathbf{R}f_*$ preserves $\A^1$-local objects and thus induces the composition $\mathbf{R}f_* \circ L_{\A^1}$ induces a functor $\mathbf{R}^{\A^1}f_*: \sH(F) \to \sH(k)$ (see \cite[pages 92 and 108]{Morel-Voevodsky}). It follows from \cite[page 109, Proposition 2.12]{Morel-Voevodsky} that for any simplicial sheaf $\sX$ on $Sm/F$, the canonical morphism 
\[
\mathbf{R}f_*(\sX) \to \mathbf{R}^{\A^1}f_*(\sX) = \mathbf{R}f_* \circ L_{\A^1}(\sX)
\]
is an $\A^1$-weak equivalence. This induces an isomorphism 
\[
L_{\A^1} \circ \mathbf{R}f_* (\sX) \to \mathbf{R}f_* \circ L_{\A^1}(\sX).  
\]

\begin{notation}
Let $F$ be a finite field extension of $k$.  For the finite map $f: \Spec F \to \Spec k$, we will denote $\mathbf{R}f_*$ by $R_{F/k}$. 
\end{notation}

The following is a straightforward consequence of the above discussion.

\begin{lemma}
\label{lemma weil restriction}
Let $F/k$ be a finite extension of fields.  For every simplicial sheaf $\sX$ over $Sm/F$, we have $R_{F/k} \pi_0^{\A^1}(\sX) = \pi_0^{\A^1}(R_{F/k}\sX)$.
\end{lemma}

\section{Algebraic groups and their \texorpdfstring{$\A^1$}{A1}-connected components}
\label{section preliminaries algebraic groups}

In this section, we briefly recall the basic definitions and properties from algebraic group theory; for details, refer to \cite{SGA3.2}, \cite{SGA3.3} and \cite[Appendix A]{Conrad-Gabber-Prasad}.

We will always work over a field $k$.  We will write $GL_n$ for the general linear group scheme and write $\G_m$ for $GL_1$.  We recall the definitions of reductive and semisimple group schemes from \cite[Expos\'e XIX, 1.6, 2.7]{SGA3.3}.  A \emph{reductive} algebraic group over $k$ is a smooth, affine $k$-group scheme with trivial unipotent radical.   A \emph{semisimple} algebraic group over $k$ is a smooth, affine $k$-group scheme with trivial radical.  Over a field, reductive algebraic groups are \emph{linear}, that is, they admit a finitely presented, closed immersion into $GL_n$ for some $n$, which is a group homomorphism.  All the reductive algebraic groups considered in what follows will be assumed to be connected.

The \emph{derived group} of $G$ \cite[Expos\'e XXII, Theorem 6.2.1(iv)]{SGA3.3}, will be denoted by $G_{\rm der}$. It is a normal, semisimple subgroup scheme of $G$ and the quotient 
\[
{\rm corad}(G) :=  G/G_{\rm der}
\]
is a $k$-torus \cite[Expos\'e XXII, 6.2]{SGA3.3} called the \emph{coradical} of $G$.

The center of a reductive group is of multiplicative type \cite[Expos\'e XII, Proposition 4.11]{SGA3.2}.  There exists a central isogeny \cite[Expos\'e XXII, Proposition 6.2.4]{SGA3.3}
\[
G_{\rm der} \times T \to G,
\]
\noindent where $T$ is a torus, the radical of $G$.  This is a faithfully flat, finitely presented morphism, whose kernel is a finite group of multiplicative type contained in the center of $G_{\rm der} \times T$.

An algebraic group is said to be \emph{almost $k$-simple} if it is smooth, connected over $k$ and admits no infinite normal $k$-subgroup \cite[p.41]{Tits-classification}. An algebraic group $G$ over $k$ is said to be \emph{absolutely almost simple} if $G_{\-k}$ is almost $\-k$-simple.  An algebraic group $G$ is said to be the \emph{almost direct product} of its algebraic subgroups $G_1, \ldots, G_n$ if the product map 
\[
G_1 \times \cdots \times G_n \to G
\]
is an isogeny.  Semisimple algebraic $k$-groups are exactly those that occur as the almost direct product of their almost $k$-simple algebraic subgroups, called the \emph{almost $k$-simple factors}.

A connected semisimple algebraic group $G$ over $k$ is said to be \emph{simply connected} if every central isogeny $G' \to G$ is an isomorphism.  Given a connected semisimple algebraic group $G$, there exists a simply connected group $G_{\rm sc}$ and a central isogeny $\pi: G_{\rm sc} \to G$.  The pair $(G_{\rm sc}, \pi)$ is unique up to unique isomorphism and its formation respects base change by field extensions.  $G_{\rm sc}$ is called the \emph{simply connected central cover} of $G$.  Every semisimple simply connected $k$-group is uniquely given by a direct product of almost $k$-simple simply connected groups. If $G$ is almost $k$-simple and simply connected, there exists a finite field extension $F/k$  and an absolutely almost simple, simply connected $F$-group $H$ such that $G = R_{F/k}(H)$ \cite[p. 41]{Tits-classification}.

\begin{definition}
\label{definition isotropic anisotropic}
A reductive algebraic group $G$ over a field $k$ is called \emph{isotropic} if $G$ contains a non-central $k$-subgroup scheme isomorphic to $\G_m$.  A reductive algebraic group $G$ over a field $k$ is called \emph{anisotropic} if it contains no subgroup isomorphic to $\G_m$.
\end{definition}

We now describe $\A^1$-connected components of algebraic groups over a field $k$.  By \cite[Theorem 4.18]{Choudhury}, for any algebraic group $G$, the sheaf $\pi_0^{\A^1}(G)$ is $\A^1$-invariant.  Putting this together with \cite[Theorem 1]{Balwe-Hogadi-Sawant}, we obtain the following description.

\begin{proposition}
For an algebraic group $G$ over a field $k$, we have 
\[
\pi_0^{\A^1}(G) \simeq \underset{n}{\varinjlim}~\sS^n(G).
\]
\end{proposition}

We end this section with an explicit description of $\A^1$-connected components of semisimple, simply connected groups in terms of other classical invariants of algebraic groups.  This description is a straightforward consequence of the results of \cite{Balwe-Sawant}.  We first recall the definitions of \emph{$R$-equivalence} and \emph{Whitehead groups}.

\begin{definition}
\label{definition R-equivalence}
Let $G$ be an algebraic group over a field $k$.  Two $k$-rational points $x, y$ of $G$ are said to be \emph{$R$-equivalent} if there is a rational map $f: \P^1_k \dashrightarrow G$ defined at $0$ and $1$ such that $f(0)=x$ and $f(1)=y$.
\end{definition}

The notion of $R$-equivalence was first studied in the context of algebraic groups in \cite{Colliot-Thelene-Sansuc-1979}.  For the basic properties regarding $R$-equivalence, also see \cite[Section II]{Gille-IHES}, \cite{Gille} and \cite[Chapter 6]{Voskresenskii}.  It can be shown that $R$-equivalence gives an equivalence relation on $G(k)$.  It is easy to see that elements of $G(k)$ that are $R$-equivalent to the identity form a normal subgroup of $G(k)$.  The quotient of $G(k)$ by this normal subgroup is denoted by $G(k)/R$ and called the \emph{group of $R$-equivalence classes} of $G$ over $k$.  

\begin{definition}
\label{definition R-trivial}
We say that an algebraic group $G$ over a field $k$ is \emph{$R$-trivial} if the group $G(F)/R := G_F(F)/R$ is trivial, for every field extension $F/k$.  
\end{definition}

\begin{definition}
\label{definition Whitehead group}
For an algebraic group $G$ over a field $k$ and a field extension $F$ of $k$, let $G(F)^+$ be the normal subgroup of $G(F)$ generated by the subsets $U(F)$ where $U$ varies over all $F$-subgroups of $G$ which are isomorphic to the additive group $\G_a$.  The group
\[
W(F,G):= G(F)/G(F)^+
\]
\noindent is called the \emph{Whitehead group} of $G$ over $F$.
\end{definition}

Evidently, there is a canonical surjection 
\begin{equation}
\label{equation surjections1}
W(k,G) \to G(k)/R,
\end{equation}
for any algebraic group $G$ over a field $k$.  This surjection is an isomorphism if $G$ is semisimple, simply connected, absolutely almost simple and isotropic (see \cite[Th\'eor\`eme 7.2]{Gille}, for example).  The above surjection is not an isomorphism in general for non-simply connected groups.

\begin{theorem}
\label{theorem IMRN}
Let $G$ be a semisimple, simply connected group over an infinite perfect field $k$.  Let $F$ be a perfect field extension of $k$.  Then there is a canonical isomorphism 
\[
\pi_0^{\A^1}(G)(F) \to G(F)/R. 
\]
\end{theorem}
\begin{proof}
First assume that $G$ is a semisimple, simply connected and absolutely almost simple group over $k$.  If $G$ is isotropic, then by \cite[Theorem 4.3.1]{Asok-Hoyois-Wendt-2}, $\Sing G$ is $\A^1$-local and it follows that $\pi_0^{\A^1}(G)(F) \simeq G(F)/R$, for every field extension $F/k$ (see \cite[Theorem 3.4]{Balwe-Sawant}).  If $G$ is anisotropic, this is \cite[Theorem 4.2]{Balwe-Sawant} (although it is stated there with the assumption that the base field is of characteristic $0$, it is easy to see that the proof works over any infinite perfect field).

Now, let $G$ be an arbitrary semisimple, simply connected group.  There exist almost $k$-simple algebraic groups $H_1, \ldots, H_r$ such that $G \simeq H_1 \times \cdots \times H_r$.  For each $i$, there exists a finite field extension $k_i/k$ and an absolutely almost simple group $G_i$ such that $R_{k_i/k}(G_i) \simeq H_i$.  Note that for any finitely generated field extension $F/k$, we have
\[
\pi_0^{\A^1}(H_i)(F) \simeq \pi_0^{\A^1}(R_{k_i/k}(G_i))(F) \simeq R_{k_i/k}(\pi_0^{\A^1}(G_i))(F) \simeq \pi_0^{\A^1}(G_i)(F \otimes_k k_i) 
\]
However, since $F \otimes_k k_i$ is a product of fields, by the special case of absolutely almost simple groups explained above, we have
\[
\pi_0^{\A^1}(G_i)(F \otimes_k k_i) \simeq G_i(F \otimes_k k_i)/R
\]
Since
\[
G_i(F \otimes_k k_i)/R \simeq R_{k_i/k}(G_i)(F)/R \simeq H_i(F)/R,
\]
for every $i$, we conclude that
\[
\pi_0^{\A^1}(G)(F) \simeq \prod_{i=1}^{r} ~ \pi_0^{\A^1}(H_i)(F) {\simeq} \prod_{i=1}^{r} ~ H_i(F)/R \simeq G(F)/R.
\]
\end{proof}

This immediately implies the failure of $\A^1$-locality of the singular construction on $G$ satisfying the hypotheses of Theorem \ref{theorem IMRN} and having at least one anisotropic factor.

\begin{corollary}
\label{corollary IMRN}
Let $G$ be a semisimple, simply connected group over an infinite perfect field $k$.  If $G$ has an anisotropic almost $k$-simple factor, then $\Sing G$ is not $\A^1$-local.
\end{corollary}
\begin{proof}
Since $G$ is a direct product of its almost $k$-simple factors, we are reduced to the case of an anisotropic, semisimple, almost $k$-simple group.  Since $G$ is reductive over a perfect field $k$, it is unirational over $k$ (see \cite[Theorem 18.2]{Borel}).  Therefore, there exists a pair of distinct $R$-equivalent elements in $G(k)$.  Since $G$ is anisotropic, we have $G(k) \simeq \sS(G)(k)$, by \cite[Lemma 3.7]{Balwe-Sawant}.  Thus, the map $\sS(G)(k) \to \pi_0^{\A^1}(G)(k)$ is not a bijection. This shows that $\Sing G$ cannot be $\A^1$-local.
\end{proof}

\section{Failure of \texorpdfstring{$\A^1$}{A1}-locality of the singular construction and consequences}
\label{section failure}

\subsection{\texorpdfstring{$\A^1$}{A1}-locality of \texorpdfstring{$\Sing$}{Sing*}}
\label{subsection failure A1-locality}

The following isotropy hypothesis on reductive algebraic $k$-groups was introduced in \cite{Raghunathan-1989}.  This is the isotropy hypothesis $(\ast)$ from the introduction.

\begin{hypothesis}
\label{hypothesis isotropy}
Each of the almost $k$-simple components of $G_{\rm der}$ contains a $k$-subgroup scheme isomorphic to $\G_m$. 
\end{hypothesis}

\begin{remark}
We caution the reader that reductive groups satisfying Hypothesis \ref{hypothesis isotropy} are called \emph{isotropic reductive groups} in \cite{Asok-Hoyois-Wendt-2}.  However, in this paper, we stick to the classical definitions and terminology \cite[V.20.1]{Borel}.
\end{remark}

In this section, we show that if a reductive algebraic group $G$ over an infinite perfect field does not satisfy Hypothesis \ref{hypothesis isotropy}, then the Morel-Voevodsky singular construction $\Sing(G)$ is not $\A^1$-local.  A key role in the proof will be played by the \emph{fppf} classifying space $B_{\rm fppf} G$ of a reductive group $G$.  We begin by briefly introducing this object.

\begin{definition}
\label{definition t-local replacement}
Let $t$ be a Grothendieck topology on a small category $\mathbf{C}$.  We say that a simplicial sheaf $\sF$ on $\mathbf{C}$ is \emph{$t$-local} if, for every $X \in \mathbf{C}$ and every $t$-covering sieve $\sU$ of $X$, the restriction map
\[
\sF(X) \to \underset{(Y\to X) \in \sU}{\rm holim}~ \sF(Y)
\]
is a weak equivalence.   
\end{definition}

The \emph{\v{C}ech $t$-local injective model structure} on this category is the left Bousfield localization of the injective model structure with respect to the set of maps $\{\sU \to X\}$, where $X$ runs over all objects of $\mathbf{C}$ and $\sU$ runs over all covering sieves of $X$. In this model structure, an object $\sF$ is fibrant if and only if it is $t$-local and also fibrant with respect to the injective model structure.

It can be proved (see the argument in \cite[Example A.10]{Dugger-Hollander-Isaksen}) that in the case of the category $Sm/k$, the \v{C}ech Nisnevich-local injective model structure is the same as local injective model structure described in Section \ref{section preliminaries A1-connectedness}. 

We now apply this notion to the category $Sch/k$ of schemes of finite type over $k$ with the \emph{fppf} topology. Thus, we have the model category of simplicial sheaves on $Sch/k$ with the \v{C}ech \emph{fppf}-local injective model structure. Let $\mathbf{R}_{\rm fppf}$ denote the fibrant replacement functor for this model structure. The inclusion functor $i: Sm/k \to Sch/k$ induces a restriction functor $i^*$ from the category of simplicial fppf-sheaves on $Sch/k$ to the category of simplicial Nisnevich sheaves on $Sm/k$. 

For a group sheaf $G$, we will denote by $BG$ the pointed simplicial sheaf whose $n$-simplices are $G^n$ with usual face and degeneracy maps.  

\begin{definition}
Let $G$ be an $fppf$-sheaf of groups on $Sch/k$. Then we define $B_{\rm fppf}G$ to be the simplicial Nisnevich sheaf defined by 
\[
B_{\rm fppf}G := i^* \circ \mathbf{R}_{\rm fppf}(BG). 
\]
\end{definition}

We denote by $(Sm/k)_{\rm fppf}$ the site of faithfully flat, finitely presented smooth schemes over $k$ which are separated and of finite type.  We will use simplicial and $\A^1$-fiber sequences of simplicial \emph{fppf}-sheaves of sets.  Following \cite[\textsection 2]{Asok-Hoyois-Wendt-2}, by a \emph{simplicial fiber sequence} of pointed simplicial presheaves, we mean a homotopy Cartesian square in which either the top-right or bottom-left corner is a point.  One defines an \emph{$\A^1$-fiber sequence} similarly with appropriate modifications, see \cite[\textsection 2.3, Definition 2.9]{Asok}.  As in topology, $\A^1$-fiber sequences of pointed simplicial sheaves induce long exact sequences of $\A^1$-homotopy sheaves.  In what follows, we will suppress the basepoints for the sake of brevity.

The following lemma will be very useful in the proof of our main theorem.  

\begin{lemma}
\label{lemma classifying spaces}
Let $G$ be an algebraic group of multiplicative type over a field $k$.  Then the classifying space $B_{\rm fppf} G$ is $\A^1$-local. 
\end{lemma}
\begin{proof}
We imitate the proof of \cite[\textsection 4.3, Proposition 3.1]{Morel-Voevodsky}.  We abuse the notation and continue to denote by $B_{\rm fppf} G$ the restriction to $(Sm/k)_{\rm Nis}$ of the \emph{fppf}-local replacement of the simplicial presheaf $BG$. 

Since $B_{\rm fppf} G$ is etale-local in the sense of Definition \ref{definition t-local replacement} (that is, $B_{\rm fppf} G$ satisfies \'etale descent), it suffices to show that the map
\begin{equation}
\label{equation lemma multiplicative type}
(B_{\rm fppf} G)(S) \to (B_{\rm fppf} G)(\A^1_S)
\end{equation}
induced by the projection $\A^1_S \to S$ is a weak equivalence for every $S$ which is the strict henselization of a local ring of a smooth scheme over $k$.  In order to show this, it suffices to show that the map induced by the map \eqref{equation lemma multiplicative type} on every $\pi_i$ is a bijection. Since $\pi_i(B_{\rm fppf}G)$ is trivial for $i>1$, it suffices to examine the map on $\pi_i$ for $i = 0$ and $1$. Since $G$ is an algebraic group of multiplicative type, it follows that $G$ is diagonalizable over $S$ \cite[Proposition B.3.4]{Conrad-SGA}.  Hence, $G_S$ is a product of group schemes of the form $\G_m$ or $\mu_n$ over $S$, for a natural number $n$.  So without loss of generality, we may assume that $G=\G_m$ or $G=\mu_n$.   

By \cite[Lemma 2.2.2]{Asok-Hoyois-Wendt-2}, $\pi_0(B_{\rm fppf} G)(-) \simeq H^1_{\rm fppf}(-, G)$.  Hence the map induced by \eqref{equation lemma multiplicative type} on $\pi_0$ is the map
\[
H^1_{\rm fppf}(-, G) \to H^1_{\rm fppf}(-, G)
\]
induced by the projection $\A^1_S \to S$.  This map is a bijection by the $\A^1$-invariance of Picard group of schemes (over any normal base-scheme).   

It remains to verify that the map \eqref{equation lemma multiplicative type} is an isomorphism on $\pi_1$'s at the base point.  However, this is just the map $G(S) \to G(\A^1_S)$ (induced by the projection $\A^1_S \to S$), which is clearly an isomorphism, since $S$ is reduced and since $G$ is either $\G_m$ or a finite group. 
\end{proof}

We are now set to prove the main reduction step in the proof of our main theorem.

\begin{proposition}
\label{proposition central isogeny}
Let $G' \to G$ be a central isogeny of reductive algebraic groups.  Suppose that $\Sing G'$ is not $\A^1$-local.  Then $\Sing G$ cannot be $\A^1$-local. 
\end{proposition}
\begin{proof}
Suppose, if possible, that $\Sing(G)$ is $\A^1$-local.  Let $\mu$ denote the kernel of the given central isogeny $G' \to G$. Then we have a sequence
\[
G' \to G \to \mathbf{R}_{\rm fppf}(B\mu) 
\]
which is a fiber sequence in the model category of simplicial \emph{fppf}-sheaves on $Sch/k$ (with the \v{C}ech $fppf$-local injective model structure).  The restriction functor $i^*$, from the category of simplicial sheaves on $Sch/k$ to the category of simplicial sheaves on $Sm/k$, preserves objectwise fiber sequences.  An objectwise fiber sequence is a fiber sequence in $\sH_s(k)$. Thus, we have a simplicial fiber sequence
\[
G' \to G \to B_{\rm fppf}\mu.
\]
Note that $\mu$ is a group of multiplicative type, being contained in the center of the reductive group $G'$.  Since $\mu$ is $\A^1$-rigid and since $B_{\rm fppf}\mu$ is $\A^1$-local by Lemma \ref{lemma classifying spaces}, it follows that $\pi_0^{\A^1}(\mu) \simeq \mu$ is a strongly $\A^1$-invariant sheaf, in the sense of \cite[Definition 1.7]{Morel}.  Therefore, by \cite[Theorem 6.50]{Morel}, the simplicial fiber sequence 
\[
G' \to G \to B_{\rm fppf}\mu
\]
is also an $\A^1$-fiber sequence.  The associated long exact sequence of homotopy groups gives us the following commutative diagram with exact rows, for every $i\geq 0$:
\begin{SMALL}
\begin{equation}
\label{equation commutative diagram}
\minCDarrowwidth 10pt
\begin{CD}
\pi_{i+1}^s(\Sing G) @>>> \pi_{i+1}^s(\Sing B_{\rm fppf}\mu) @>>> \pi_{i}^s(\Sing G') @>>> \pi_{i}^s(\Sing G) @>>> \pi_{i}^s(\Sing B_{\text{fppf}}\mu) \\
@VV{\simeq}V @VV{\simeq}V @VVV @VV{\simeq}V @VV{\simeq}V \\
\pi_{i+1}^{\A^1}(G) @>>> \pi_{i+1}^{\A^1}(B_{\rm fppf}\mu) @>>> \pi_i^{\A^1}(G') @>>> \pi_i^{\A^1}(G) @>>> \pi_i^{\A^1}(B_{\text{fppf}}\mu) .
\end{CD} 
\end{equation}
\end{SMALL}
Here the first row is obtained as follows: since $\pi_0(B_{\rm fppf} \mu)(-) \simeq H^1_{\rm fppf}(-, \mu)$ is an $\A^1$-invariant presheaf, the functor $\Sing$ preserves simplicial fiber sequences \cite[Proposition 2.1.1]{Asok-Hoyois-Wendt-2}; we then take the associated long exact sequence of simplicial homotopy groups.  All the vertical maps are induced by the natural transformation of functors $\Sing \to L_{\A^1}$.  In the diagram \eqref{equation commutative diagram}, the second and the last vertical arrows are isomorphisms since $B_{\rm fppf}\mu$ is $\A^1$-local (Lemma \ref{lemma classifying spaces}); and the first and fourth vertical arrows are isomorphisms since $\Sing(G)$ is $\A^1$-local.  It follows from five lemma that the map $\pi_{i}^s(\Sing G') \to \pi_{i}^{\A^1}(G')$ is an isomorphism for all $i\geq 0$.  
This shows that the natural map $\Sing G' \to L_{\A^1}(G')$ is a weak equivalence, by the $\A^1$-Whitehead theorem \cite[\textsection 3, Proposition 2.14, p. 110]{Morel-Voevodsky}.  Consequently, $\Sing G'$ is $\A^1$-local, contradicting the hypothesis.
\end{proof}

We now prove the main theorem of this section.

\begin{theorem}
\label{theorem failure of A1-locality}
Let $G$ be a reductive algebraic group over an infinite perfect field $k$ that does not satisfy Hypothesis \ref{hypothesis isotropy}.  Then $\Sing G$ is not $\A^1$-local.
\end{theorem}
\begin{proof}
Since $G$ is reductive, there exists a central isogeny
\[
G_{\rm der} \times T \to G,
\]
\noindent where $G_{\rm der}$ is a semisimple group (the derived group of $G$) and $T$ is a torus (the radical of $G$).  By Proposition \ref{proposition central isogeny}, it suffices to show that $\Sing(G_{\rm der} \times T)$ is not $\A^1$-local.  Note that $\Sing$ commutes with products and that $\Sing T$ is $\A^1$-local.  Therefore, we are reduced to showing that $\Sing G_{\rm der}$ is not $\A^1$-local. 

Let $G_{\rm sc}$ denote the simply connected cover of $G_{\rm der}$.  There exists a central isogeny $G_{\rm sc} \to G_{\rm der}$.  Again by Proposition \ref{proposition central isogeny}, it suffices to prove that $\Sing(G_{\rm sc})$ is not $\A^1$-local.  Let $G_1, \ldots, G_n$ be the almost $k$-simple factors of $G_{\rm sc}$; we have an isomorphism $G_1 \times \cdots \times G_n \stackrel{\sim}{\to} G_{\rm sc}$.  If $G$ does not satisfy Hypothesis \ref{hypothesis isotropy}, then $G_{\rm der}$ has at least one anisotropic almost $k$-simple factor.  Therefore, there exists $i \in \{1, \ldots, n\}$ such that $G_i$ is anisotropic.  By Corollary \ref{corollary IMRN}, we conclude that $\Sing G_i$ is not $\A^1$-local.  Hence, $\Sing(G_1 \times \cdots \times G_n) \simeq \Sing G_{\rm sc}$ cannot be $\A^1$-local.  Thus, if $G$ does not satisfy Hypothesis \ref{hypothesis isotropy}, then $\Sing G$ cannot be $\A^1$-local.
\end{proof}

We end the section by recording a proof of Theorem \ref{intro theorem 1}, stated in the intoduction, by putting together Theorem \ref{theorem failure of A1-locality} and relevant results from \cite{Asok-Hoyois-Wendt-2}.

\begin{theorem}
\label{theorem equivalence}
Let $G$ be a reductive algebraic group over an infinite perfect field $k$.  Then the following conditions are equivalent: 
\begin{itemize}
\item[(1)] $\Sing G$ is $\A^1$-local;
\item[(2)] $G$ satisfies the isotropy hypothesis $(\ast)$ (Hypothesis \ref{hypothesis isotropy});
\item[(3)] The presheaf $H^1_{\rm Nis}(-, G)$ is $\A^1$-invariant on smooth affine schemes over $k$.
\end{itemize}

\end{theorem}
\begin{proof}
The implication $(1) \Rightarrow (2)$ follows from Theorem \ref{theorem failure of A1-locality}.  The implication $(2) \Rightarrow (3)$ is proved in \cite[Theorem 3.3.3]{Asok-Hoyois-Wendt-2}, whereas the implication $(3) \Rightarrow (1)$ is proved in \cite[Theorem 2.3.2]{Asok-Hoyois-Wendt-2}.
\end{proof}

\subsection{Failure of affine homotopy invariance for \texorpdfstring{$G$}{G}-torsors}
\label{subsection failure A1-invariance}

In \cite[Theorem 3.3.6]{Asok-Hoyois-Wendt-2}, it is shown that if $G$ is a reductive algebraic group over an infinite field $k$ satisfying Hypothesis \ref{hypothesis isotropy} and $A$ is a smooth affine $k$-algebra, then the map 
\[
H^1_{\rm Nis}(\Spec A, G) \to H^1_{\rm Nis}(\Spec A[t_1, \ldots, t_n], G)
\]
induced by the projection $\Spec A[t_1, \ldots, t_n] \to \Spec A$ is a bijection for all $n \geq 0$.  In view of the Grothendieck-Serre conjecture (see \cite{Colliot-Thelene-Ojanguren-1995}, \cite{Raghunathan-1994}, \cite{Fedorov-Panin}), Nisnevich locally trivial $G$-torsors are Zariski locally trivial, where $G$ is a connected reductive group over an infinite perfect field $k$.  Therefore, \cite[Theorem 3.3.6]{Asok-Hoyois-Wendt-2} can be seen as a generalization of the results of Lindel \cite{Lindel} (the case $G = GL_n$) and Raghunathan \cite{Raghunathan-1989} (the case where $G$ satisfies Hypothesis \ref{hypothesis isotropy} and $A=k$).  Counterexamples to affine homotopy invariance of $G$-torsors were found in case the group $G$ does not satisfy Hypothesis \ref{hypothesis isotropy} by Ojanguren-Sridharan \cite{Ojanguren-Sridharan} and Parimala \cite{Parimala}.  A general result about failure of affine homotopy invariance is due to Raghunathan \cite[Theorem B]{Raghunathan-1989}, where it is shown that if $G$ is an anisotropic, absolutely almost simple group not of type $F_4$ or $G_2$ and satisfying a technical condition (there exists a group $G'$ in the central isogeny class of $G$ and an embedding of $G'$ in a connected reductive group $H$ as a closed normal subgroup such that $H$ is a $k$-rational variety and such that $H/G'$ is a torus) which holds for groups of classical type, then there are infinitely many mutually non-isomorphic $G$-bundles on $\A^2_k$ that are not extended from $\Spec k$. 

A straightforward application of Theorem \ref{theorem equivalence} shows that torsors for reductive groups not satisfying Hypothesis \ref{hypothesis isotropy} fail to be $\A^1$-invariant on smooth affine schemes over an infinite perfect field.  In the case of semisimple, simply connected, absolutely almost simple anisotropic groups, this was shown by Asok, Hoyois and Wendt in \cite[Proposition 3.3.7]{Asok-Hoyois-Wendt-2} using \cite[Corollary 3]{Balwe-Sawant}.  Theorem \ref{theorem failure of A1-locality} generalizes \cite[Corollary 3]{Balwe-Sawant} to all reductive algebraic groups not satisfying Hypothesis \ref{hypothesis isotropy} and hence generalizes \cite[Proposition 3.3.7]{Asok-Hoyois-Wendt-2} to all such groups using the same method.  We end this section by formally stating the result for the sake of completeness.

\begin{proposition}
\label{proposition failure of A1-invariance}
Let $G$ be a reductive algebraic group over an infinite perfect field $k$, which does not satisfy Hypothesis \ref{hypothesis isotropy}.  Then the presheaf $H^1_{\rm Nis}(-, G)$ cannot be $\A^1$-invariant on smooth affine schemes over $k$. 
\end{proposition}

\section{Characterization of \texorpdfstring{$\A^1$}{A1}-connectedness in reductive groups and applications}
\label{section characterization of A1-connectedness}

\subsection{\texorpdfstring{$\A^1$}{A1}-connected reductive algebraic groups}

In this section, we characterize $\A^1$-connected reductive algebraic groups over a field of characteristic $0$.  We first treat the case where the base field is algebraically closed.

\begin{theorem}
\label{theorem algebraically closed}
Let $G$ be a reductive algebraic group over an algebraically closed field $k$. Then $G$ is $\A^1$ -connected if and only if it is semisimple and simply connected.
\end{theorem}

\begin{proof}
First, we assume that $G$ is semisimple and simply connected.  Since $G$ is semisimple and $k$ is algebraically closed, $G$ is an almost direct product of absolutely almost simple groups $H_1, \ldots, H_r$ over $k$.  Since $G$ is simply connected, it has no nontrivial isogenies, thus giving an isomorphism $\prod_{i=1}^{r} ~ H_i \xrightarrow{\sim}G$.

Since $k$ is algebraically closed, each of the $H_i$'s is a rational variety \cite[14.14, Remark]{Borel} and hence, is $R$-trivial.  By Theorem \ref{theorem IMRN}, we have $$\pi_0^{\A^1}(H_i)(F) \simeq H_i(F)/R \simeq \ast,$$ for every field extension $F$ of $k$.  By \cite[Lemma 6.1.3]{Morel-connectivity}, it follows that $G$ is $\A^1$-connected.

We now prove the converse. Let $G$ be an $\A^1$-connected reductive algebraic group. We have the exact sequence 
\[
1 \to G_{\rm der} \to G \to \rm corad (G) \to 1,
\]
where $G_{\rm der}$ is a semisimple group (the derived group of $G$) and ${\rm corad}(G)$ is a torus (the coradical of $G$).  Since $k$ is algebraically closed, $G$ clearly satisfies Hypothesis \ref{hypothesis isotropy}.  Therefore, by \cite[Theorem 4.3.1]{Asok-Hoyois-Wendt-2}, $\Sing G$ is $\A^1$-local and we have $\sS(G) \simeq \pi_0^{\A^1}(G)$.  Since tori are $\A^1$-rigid, the surjective map $G(k) \to {\rm corad}(G)(k)$ induces a surjective map $\sS(G)(k) \to {\rm corad}(G)(k)$.  Since $\sS(G)= \ast$ by hypothesis, it follows that ${\rm corad}(G)$ is trivial.  Thus, $G_{\rm der} \simeq G$, that is, $G$ is semisimple.

Let $G_{\rm sc}$ denote the simply connected central cover of $G$ so that we have a central isogeny $G_{\rm sc} \to G$, whose kernel will be denoted by $\mu$. As we noted before in the proof of Proposition \ref{proposition central isogeny}, the simplicial fiber sequence 
\[
G_{\rm sc} \to G \to B_{\rm fppf} \mu
\]
is also an $\A^1$-fiber sequence. Thus, the long exact sequence of $\A^1$-homotopy gives us the following exact sequence of pointed sheaves:
\[
\cdots \to \pi_1^{\A^1}(B_{\rm fppf} \mu) \to \pi_0^{\A^1}(G_{\rm sc}) \to \pi_0^{\A^1}(G) \to \pi_0^{\A^1}(B_{\rm fppf} \mu). 
\]
Since $G_{\rm sc}$ is a split semisimple simply connected group, we have $\pi_0^{\A^1}(G_{sc})(F) \simeq \sS(G_{\rm sc})(F) \simeq W(F, G_{\rm sc}) = \ast$, for every finitely generated field extension $F$ of $k$, by \cite[1.1.2]{Tits-Bourbaki}. 

Let $k(G)$ and $k(G_{\rm sc})$ denote the function fields of $G$ and $G_{\rm sc}$ respectively. Let $\eta: \Spec k(G) \to G$ denote the generic point of $G$.  By Lemma \ref{lemma classifying spaces}, $B_{\rm fppf} \mu$ is $\A^1$-local and so $\pi_0^{\A^1}(B_{\rm fppf} \mu)$ is the sheafification of the pointed presheaf $H^1_{\rm fppf}(-, \mu)$ (where the base-point corresponds to the trivial torsor).  Thus, $\pi_0^{\A^1}(B_{\rm fppf} \mu)(k(G)) = H^1_{\rm fppf}(k(G), \mu)$ and the image of $\eta$ in $H^1_{\rm fppf}(k(G), \mu)$ under the composition
\[
G(k(G)) \to \pi_0^{\A^1}(G)(k(G)) \to \pi_0^{\A^1}(B_{\rm fppf} \mu)(k(G))
\]
corresponds to the class of the $\mu$-torsor $G_{\rm sc} \times_{G, \eta} \Spec k(G) \to \Spec k(G)$.  Since $\pi_0^{\A^1}(G) = \ast$, the morphism of pointed sheaves $\pi_0^{\A^1}(G) \to \pi_0^{\A^1}(B_{\rm fppf} \mu)$ in the above $\A^1$-fiber sequence is trivial.  Hence, the image of $\eta$ under the composition 
\[
G(k(G)) \to \pi_0^{\A^1}(G)(k(G)) \to \pi_0^{\A^1}(B_{\rm fppf} \mu)(k(G))
\]
is equal to the base-point of $\pi_0^{\A^1}(B_{\rm fppf} \mu)(k(G))$, that is, the class of the trivial torsor.  Therefore, the $\mu$-torsor $G_{\rm sc} \times_{G, \eta} \Spec k(G) \to \Spec k(G)$ is trivial and thus admits a section.  Hence, the morphism $\eta: \Spec k(G) \to G$ can be lifted to a morphism $\eta': \Spec k(G_{\rm sc}) \to G_{\rm sc}$. The image of this morphism must be the generic point of $G_{\rm sc}$.  As a result, $\eta'$ induces a morphism $\Spec k(G) \to \Spec k(G_{\rm sc})$ which is a section of the morphism $\Spec k(G_{\rm sc}) \to \Spec k(G)$ induced by the isogeny $G_{\rm sc} \to G$. This gives us a sequence $k(G) \to k(G_{\rm sc}) \to k(G)$ of homomorphisms of fields such that the composition is the identity homomorphism on $k(G)$. Thus, we see that the homomorphism $k(G) \to k(G_{\rm sc})$ induced by the isogeny $G_{\rm sc} \to G$ is an isomorphism. We conclude that the isogeny $G_{\rm sc} \to G$ is a finite morphism of degree $1$ and hence it is an isomorphism. Thus, we see that $G$ is simply connected.
\end{proof}

Using Theorem \ref{theorem algebraically closed} and Lemma \ref{lemma weil restriction}, we now treat the general case.

\begin{theorem}
\label{theorem reductive connectedness}
Let $G$ be a reductive algebraic group over a field $k$ of characteristic $0$.  Then the following are equivalent: 
\begin{itemize}
\item[(1)]  $G$ is $\A^1$-connected;
\item[(2)]  $G$ is semisimple, simply connected and the almost $k$-simple factors of $G$ are $R$-trivial. 
\end{itemize}
\end{theorem}
\begin{proof}
(1) $\Rightarrow$ (2):   Let $\-k$ denote an algebraic closure of $k$.  By Theorem \ref{theorem algebraically closed}, $G_{\-k}$ is semisimple and simply connected.  Therefore, $G$ is semisimple and simply connected.  Now, there exist almost $k$-simple algebraic groups $H_1, \ldots, H_r$ such that $G \simeq H_1 \times \cdots \times H_r$.  Since $G$ is $\A^1$-connected, each of the $H_i$'s is $\A^1$-connected.  For each $i$, there exists a finite field extension $k_i/k$ and an absolutely almost simple group $G_i$ such that $R_{k_i/k}(G_i) \simeq H_i$.  Note that
\[
H_i(F)/R = R_{k_i/k}(G_i)(F)/R \simeq G_i(F \tensor_k  k_i)/R
\]
By Theorem \ref{theorem IMRN}, we have $G_i(F \tensor_k  k_i)/R \simeq \pi_0^{\A^1}(G_i)(F \tensor_k k_i)$ for each $i$, since each $G_i$ is semisimple, simply connected, absolutely almost simple.  Hence, for every $i$, we have
\[
H_i(F)/R \simeq \pi_0^{\A^1}(G_i)(F \tensor_k k_i) \simeq \pi_0^{\A^1}(R_{k_i/k}(G_i))(F) \simeq \pi_0^{\A^1}(H_i)(F) \simeq \ast.
\]

\noindent (2) $\Rightarrow$ (1):  Since $G$ is semisimple, it is an almost direct product of almost $k$-simple groups $H_1, \ldots, H_r$.  For each $i$, there exists a finite field extension $k_i/k$ and an absolutely almost simple group $G_i$ such that the Weil restriction $R_{k_i/k}(G_i)$ is isomorphic to $H_i$.  Since $G$ is simply connected, it has no nontrivial isogenies, thus giving isomorphisms
\[
\prod_{i=1}^{r} ~ R_{k_i/k}(G_i) \xrightarrow{\simeq} \prod_{i=1}^{r} ~ H_i \xrightarrow{\simeq}G.
\]
By Theorem \ref{theorem IMRN}, $\pi_0^{\A^1}(G_i)(F) = G_i(F)/R$.  Since for every $i$, the group $H_i=R_{k_i/k}(G_i)$ is $R$-trivial by hypothesis, we have
\[
\pi_0^{\A^1}(R_{k_i/k}(G_i))(F) \simeq \pi_0^{\A^1}(G_i)(F \tensor_k k_i) \simeq G_i(F \tensor_k  k_i)/R \simeq R_{k_i/k}(G_i)(F)/R \simeq \ast.
\]
By \cite[Lemma 6.1.3]{Morel-connectivity}, it follows that every $R_{k_i/k}(G_i)$ is $\A^1$-connected.  Consequently, $G$ is $\A^1$-connected.  This completes the proof of the theorem.
\end{proof}

\subsection{Abelian-ness of \texorpdfstring{$\A^1$}{A1}-connected components of certain reductive algebraic groups}

Recall from Section \ref{section preliminaries algebraic groups} that $\A^1$-connected components of semisimple, simply connected groups agree with their $R$-equivalence classes (Corollary \ref{corollary IMRN}).  In this subsection, as an application of Theorem \ref{theorem reductive connectedness}, we show that the sheaf of $\A^1$-connected components of certain reductive groups is a sheaf of abelian groups.

Let $G$ be a semisimple algebraic group over a field of characteristic $0$.  The simply connected cover $G_{\rm sc}$ of $G$ gives rise to a central isogeny
\[
G_{\rm sc} \to G,
\]
whose kernel $\mu$ is a finite abelian group.  By Lemma \ref{lemma classifying spaces}, we get an $\A^1$-fiber sequence
\[
G_{\rm sc} \to G \to B_{\text{fppf}}\mu
\]
The associated long exact sequence of $\A^1$-homotopy groups yield the following exact sequence:
\[
\cdots \to \pi_0^{\A^1}(G_{\rm sc}) \to \pi_0^{\A^1}(G) \to \pi_0^{\A^1}(B_{\rm fppf} \mu). 
\]
We first show that the morphism $\pi_0^{\A^1}(G) \to \pi_0^{\A^1}(B_{\text{fppf}} \mu)$ is a homomorphism of group sheaves.  Given a smooth scheme $U$ over $k$, and an element $s \in G(U)$, we obtain a $\mu$-torsor $G'_s:= G'\times_{G,s} U \to U$, which defines an element of $H^1_{\text{fppf}}(U,\mu) = \pi_0(B_{\text{fppf}}\mu)(U)$. Thus, we obtain a morphism of sheaves $G \to \pi_0(B_{\text{fppf}}\mu)$. This morphism factors through the quotient morphism $G \to \pi_0^{\A^1}(G)$, inducing the morphism $\pi_0^{\A^1}(G) \to \pi_0(B_{\text{fppf}}\mu)$ mentioned above. Thus, it suffices to show that the map $G(U) \to H^1_{\text{fppf}}(U,\mu)$ is a homomorphism.

For an element $s \in G(U)$, suppose $\{U_i \to U\}_{i\in I}$ is an fppf cover which trivializes the torsor $G'_s$. Thus, for every $i$, there exists an element $s'_i \in G'(U_i)$ such that the map $G'(U) \to G(U)$ maps $s'_i$ to $s|_{U_i}$. For any two indices $i,j \in I$, we write $U_{ij}:= U_i \times_U U_j$ and define $s_{ij} = (s'_i|_{U_{ij}})(s'_j|_{U_{ij}})^{-1}$. Then, the collection $(s_{ij})_{i,j}$ is a $1$-cocycle which represents the isomorphism class $[G_s]$ of $G_s$ in $H^1_{\text{fppf}}(U,\mu)$.

Given elements $s,t \in G(U)$, we may choose an fppf cover $\{U_i \to U\}_{i \in I}$ that trivializes both $G'_s$ and $G'_t$. It is easy to see (using the fact that $\mu$ is abelian) that the collection $(s_{ij}t_{ij})_{i,j}$ is a $1$-cocycle which defines a $\mu$-torsor over $U$, which we denote by $G'_{s \ast t}$. The binary operation $([G'_s],[G'_t]) \mapsto [G'_{s \ast t}]$ is precisely the one that defines the group structure on $H^1_{\text{fppf}}(U,\mu)$. If $st$ denotes the product of $s$ and $t$ in $G(U)$, we wish to prove that $[G'_{st}] = [G'_{s \ast t}]$. In other words, we wish to prove that the torsor $G_{st}$ can be represented by the $1$-cocycle $(s_{ij}t_{ij})_{i,j}$. 

For every $i \in I$, the element $s'_i t'_i \in G'(U_i)$ maps to $s|_{U_i} t|_{U_i} \in G(U_i)$. Thus, the cover $\{U_i \to U\}_{i \in I}$ trivializes the torsor $G'_{st}$. So, the class $[G'_{st}]$ is represented by the $1$-cocycle $\{u_{ij}\}_{i,j}$ where 
\begin{align*}
u_{ij} & := (s'_i|_{U_{ij}})(t'_i|_{U_{ij}})(t'_j|_{U_{ij}})^{-1}(s'_j|_{U_{ij}})^{-1} \\
       & = (s'_i|_{U_{ij}})(s'_j|_{U_{ij}})^{-1}(t'_i|_{U_{ij}})(t'_j|_{U_{ij}})^{-1} \\
       & = s_{ij}t_{ij},
\end{align*}
where the second equality follows from the fact that $(t'_i|_{U_{ij}})^{-1}(t'_j|_{U_{ij}})^{-1}$ lies in $\mu(U_{ij})$ and hence commutes with $(s'_j|_{U_{ij}})^{-1}$. 

We are now ready to show the main result of this subsection.

\begin{proposition}
\label{proposition pi0 abelian}
Let $G$ be a semisimple algebraic group over an infinite perfect field such that its simply connected cover $G_{\rm sc}$ is $R$-trivial.  Then $\pi_0^{\A^1}(G)$ is a sheaf of abelian groups.
\end{proposition}
\begin{proof}
The simply connected cover $G_{\rm sc}$ of $G$ gives rise to an $\A^1$-fiber sequence
\[
G_{\rm sc} \to G \to B_{\text{fppf}}\mu
\]
as described in the discussion above.  The associated long exact sequence of $\A^1$-homotopy groups yield the following exact sequence:
\[
\cdots \to \pi_0^{\A^1}(G_{\rm sc}) \to \pi_0^{\A^1}(G) \to \pi_0^{\A^1}(B_{\rm fppf} \mu). 
\]
Since $G_{\rm sc}$ is $R$-trivial, so are its almost $k$-simple components.  By Theorem \ref{theorem reductive connectedness}, we then have $\pi_0^{\A^1}(G_{\rm sc}) = \ast$.  Hence we have an injection of sheaves $\pi_0^{\A^1}(G) \hookrightarrow \pi_0^{\A^1}(B_{\rm fppf} \mu)$.  Since $\pi_0(B_{\rm fppf} \mu)(-) = H^1_{\rm fppf}(-, G)$ by \cite[Lemma 2.2.2]{Asok-Hoyois-Wendt-2} and since $B_{\text{fppf}}\mu$ is $\A^1$-local, we conclude that $\pi_0^{\A^1}(B_{\rm fppf} \mu)$ is a sheaf of abelian groups.  Since $\pi_0^{\A^1}(G) \to \pi_0^{\A^1}(B_{\text{fppf}} \mu)$ is a homomorphism of group sheaves, we conclude that $\pi_0^{\A^1}(G)$ is a sheaf of abelian groups.
\end{proof}

\begin{remark}
\label{remark pi0(G) abelian}
It is an open question whether $\pi_0^{\A^1}(G)$ is always a sheaf of abelian groups, for a reductive group $G$ over a field.
\end{remark}

\subsection*{Acknowledgements}
We are grateful to Marc Hoyois for a very helpful correspondence; especially, for pointing out the need to use the \emph{fppf} topology and for suggesting that Lemma \ref{lemma classifying spaces} can be proven by mimicking the proof of \cite[\textsection 4.3, Proposition 3.1]{Morel-Voevodsky}.  We also thank Aravind Asok for comments and discussions and the referee for a careful reading of the paper as well as for a number of suggestions that improved the presentation.  Finally, we warmly thank Fabien Morel for stimulating discussions, suggestions and encouragement.  Part of this work was done when the second-named author was visiting Tata Institute of Fundamental Research, Mumbai, India; he thanks the institute for hospitality.

\end{document}